\documentclass{amsart}
\usepackage{mathrsfs}
\usepackage{amssymb}
\usepackage{url, array}
\usepackage[T1]{fontenc}
\usepackage{color}
\usepackage{comment}
\usepackage[all]{xy}
\usepackage{tikz}
\usepackage{mathtools}
\usepackage{hyperref}
\newtheorem{thm}{Theorem}[section]
\newtheorem{lemma}[thm]{Lemma}

\DeclareMathOperator{\id}{id}

\DeclareMathOperator{\End}{End}
\DeclareMathOperator{\Hom}{Hom}

\theoremstyle{definition}
\newtheorem*{summary*}{Summary}
\newtheorem{defi}[thm]{Definition}
\newtheorem{construction}[thm]{Construction}
\newtheorem{example}[thm]{Example}
\theoremstyle{remark}
\newtheorem*{remark*}{Remark}

\newcommand{\coasA}{\coas_A}
\newcommand{\cocomA}{\textsf{coCom}_A}

\renewcommand{\part}{\mathcal{P}}

\newcommand{\free}{\mathcal{F}}
\newcommand{\cofree}{\free^c}

\newcommand{\CC}{\mathbf{C}}
\newcommand{\PP}{\mathbf{P}}
\newcommand{\MM}{\mathbf{M}}
\newcommand{\II}{\mathbf{I}}
\newcommand{\Tt}{\mathcal{T}}
\DeclareMathOperator{\triv}{triv}

\newcommand{\coas}{\textsf{coAs}}

\title{An operadic approach to operator-valued free cumulants}
\author{Gabriel C. Drummond-Cole}
\thanks{This work was supported by IBS-R003-D1.}
\address{Center for Geometry and Physics, Institute for Basic Science (IBS), Pohang 37673, Republic of Korea}
\begin{document}
\begin{abstract}
An operadic framework is developed to explain the inversion formula relating moments and cumulants in operator-valued free probability theory. 
\end{abstract}
\maketitle
\section{Introduction}
The purpose of this paper is to provide a convenient operadic framework for the cumulants of free probability theory. 

In~\cite{DrummondColeParkTerilla:HPTI,DrummondColeParkTerilla:HPTII}, the author and his collaborators described an operadic framework for so-called Boolean and classical cumulants. In those papers, the fundamental object of study is an algebra $A$ equipped with a linear map $E$, called \emph{expectation}, to some fixed algebra $B$. The expectation is not assumed to be an algebra homomorphism; rather one measures the degree to which $E$ fails to be an algebra homomorphism with a sequence of multilinear maps $\kappa_n$ from powers of $A$ to $B$, called cumulants. The cumulants, in many cases, can be defined recursively in terms of the expectation map via a formula of the form:
\begin{equation}\label{outline of cumulants}
E(x_1\cdots x_n)= \sum \kappa_{i_1}(\cdots)\cdots \kappa_{i_k}(\cdots).
\end{equation}
Depending on what kind of probability theory is under consideration, the summation on the left may be over a different index set. See, e.g.,~\cite{Speicher:OUP,Muraki:FIQUP,HasebeSaigo:JCNI}.

In~\cite{DrummondColeParkTerilla:HPTI,DrummondColeParkTerilla:HPTII}, these recursive definitions for the collection of cumulants (in the Boolean and classical regimes, respectively) were reinterpreted as the collection of linear maps determining a coalgebraic map into a cofree object. In the Boolean case, the cofree object is the tensor coalgebra. In the classical case it is the symmetric coalgebra.

This reformulation is intended as the background for a homotopical enrichment of probability theory; adding a grading, a filtration, and a differential to this coalgebraic picture leads to a rich theory with applications to quantum field theory~\cite{Park:HTPSICIHLA}. This application is motivational and will play no role in this paper.

None of the work mentioned above treats the case of \emph{free cumulants}, arguably the most important kind of cumulant in noncommutative probability theory. When the target algebra $B$ is commutative, there is a formula similar to those above and the framework outlined above can be used directly, employing a more exotic type of coalgebra than the tensor or symmetric coalgebra. This point of view is taken in~\cite{DrummondCole:NCWCFFHPT}.

However, there is a flaw in this point of view, which is that assuming the target to be commutative is external to the theory; internally it makes perfect sense for the target itself to be noncommutative. This is called \emph{operator-valued} free probability theory because the expectation is valued in a noncommutative algebra, such as an operator algebra.

Operator-valued free cumulants, as defined by Speicher,~\cite{Speicher:CTFPAOVFPT} are somewhat more cumbersome to describe explicitly than in the commutative case using classical combinatorial methods. Consequently Speicher develops an \emph{operator-valued $R$-transform} to collect the information concisely.

In our setting, there is one evident related obstruction to extending the framework developed in~\cite{DrummondColeParkTerilla:HPTI,DrummondColeParkTerilla:HPTII} to operator-valued free cumulants. The defining formulas for classical and Boolean cumulants and for free cumulants valued in a commutative algebra share a certain property. Namely, they are \emph{string-like}, meaning that the right-hand side of Equation~(\ref{outline of cumulants}) is a product of cumulants. However, the defining formulas for operator-valued free cumulants (that is, free cumulants valued in a not necessarily commutative algebra) contain terms like
\[
\kappa_2(x_1\kappa_1(x_2)\otimes x_3).
\]
or more generally
\[
\kappa_{n_1}(x_1\kappa_{n_2}(x_2\kappa_{n_3}(\cdots)\kappa_{n_4}(\cdots),\cdots),\cdots).
\]
In a word, they are not string-like but \emph{tree-like}. 

This is precisely the issue that leads Speicher to develop the operator-valued $R$-transform. Here, this tool is avoided by using an operadic reformulation. Tree-like formulas can be obtained by passing from algebras and coalgebras, which have a string-like structure, to nonsymmetric operads and cooperads, which have a tree-like structure. The main result of this paper shows how the relationship between the moments and free cumulants, realized as cooperadic maps $M$ and $K$, is encapsulated quite simply in terms of a canonical twist:
\begin{equation*}
M=\Phi\circ K.
\end{equation*}
As phrased in this paper, the moments and cumulants are defined in some other manner and this is a theorem, but it is probably better to consider this as an alternative definition which is quite simple from the operadic viewpoint.

This reformulation is part of a campaign to explore applications of the operadic language in probability theory; the result contained herein is modest and is intended to serve as further advertisment and evidence (following~\cite{Male:DTLRMTFP,DrummondColeParkTerilla:HPTI,DrummondColeParkTerilla:HPTII,DrummondColeTerilla:CIHPT,DrummondCole:NCWCFFHPT}) of potentially deeper connections between the two areas.

It is possible that both this reformulation and those attempted in the author's previous work (cited above) are reflections of a combinatorial relationship between operads and M\"obius inversion with respect to a poset. This is not pursued further here, but see~\cite[3.3]{Mendez:SOCCS} for some discussion and further references on this topic.

The remainder of the paper is organized as follows. Section~\ref{sec: operads} reviews the parts of operadic theory that are used in the paper. Section~\ref{sec: partitions and trees} goes over the combinatorics of non-crossing partitions, and Section~\ref{sec:freeprob} applies this to define free cumulants. Finally, Section~\ref{sec: main  result} states and proves the reformulation of free cumulants in operadic terms.

\subsection{Conventions}
Everything linear occurs over a fixed ground ring. Algebras are generally not assumed to be commutative or unital. Every finite ordered set is canonically isomorphic to $[n]\coloneqq\{1,\ldots, n\}$, and this canonical isomorphism will be routinely abused. 

A graph is a finite set of vertices, a finite set of half-edges, a source map from half-edges to vertices, and an involution on the half-edges; a half-edge is a leaf it is a fixed point of the involution. A graph is connected if every two vertices can be joined by a path of half-edges connected by having the same source or via the involution. A connected graph is a tree if it has more vertices than edges. A root is a choice of leaf of a tree (this is no longer considered a leaf). The root of a vertex is the unique half-edge ``closest'' to the overall root. The root vertex is the unique vertex whose root is the overall root. A planar tree has a cyclic order on the half-edges of each vertex. 

\section{Operads and cooperads}\label{sec: operads}
Aside from some minor changes, conventions of~\cite{LodayVallette:AO} are used for operadic algebra. This section reviews standard definitions (more details can be seen in~\cite[5.9]{LodayVallette:AO}).
\begin{defi}
A \emph{collection} $M=\{M_n\}_{n\ge 0}$ is a set of modules indexed by nonnegative numbers (the index is called \emph{arity}).

Given a collection $M$, a graph \emph{decorated} by $M$ is a pair $(G,D)$ where $G$ is a graph and $D=\{D_v\}$ is a collection of elements of $M$ indexed by the vertices of $G$; for a vertex of valence $k+1$ the decoration $D_v$ should be in the module $M_k$.

There is a \emph{composition product} denoted $\circ$ on collections
\[
(M\circ N)_n=\bigoplus_k M_k\otimes \left(\bigoplus_{i_1+\cdots+i_k=n} N_{i_1}\otimes\cdots \otimes N_{i_k}\right).
\]
This product has a unit $I$, where $I_1$ is the ground ring and $I_{n\ne 1}$ is $0$, and together $\circ$ and $I$ make the category of collections into a monoidal category.
\end{defi}
\begin{defi}
A \emph{nonsymmetric operad} is a monoid $\PP$ in this monoidal category. Its data can be specified by giving a collection $P$, a \emph{composition} map $\gamma:P\circ P\to P$, and a \emph{unit} map $\eta:I\to P$ satisfying associativity and unital constraintns.

A \emph{nonsymmetric cooperad} is a comonoid $\CC$ in this monoidal category. Its data can be specified by giving a collection $C$, a \emph{decomposition} map $\Delta:C\to C\circ C$, and a \emph{counit} map $\epsilon:C\to I$ satisfying coassociativity and counital constraints.
The collection $I$ has a canonical nonsymmetric cooperad structure, denoted $\II$.
\end{defi}
In this paper everything will be nonsymmetric and the adjective will be omitted.
\begin{defi}
Let $\CC$ be a cooperad with underlying collection $C$. A \emph{coaugmentation} of $\CC$ is a map of cooperads $\eta:\II\to \CC$. 

The decomposition map $\Delta$ induces a decomposition map $\widetilde{\Delta}:C\to C\circ C$ realized by $\Delta-\eta\II\circ\id+\eta\epsilon\circ\eta\II$.

The notation $\widetilde{\Delta}^n$ is used for the map $C\to C^{\circ n+1}$ given by composition of the maps 
\[(\widetilde{\Delta}\circ\underbrace{\id\circ\cdots\circ \id}_{n-1})
:
{C}^{\circ n}\to {C}^{\circ n+1}
.\]
A coaugmented cooperad $\CC$ is \emph{conilpotent} if for every element $c\in\CC$, there is a natural number $N$ such that $\Delta^N c = (\Delta^{N-1}c)\circ \epsilon\II$.
\end{defi}
\begin{example}
\begin{itemize}
\item
The motivating example of an operad is the \emph{endomorphism operad} of a vector space $B$, denoted  $\End B$. The module $(\End B)_n$ is $\Hom(B^{\otimes n},B)$ and the image of the unit is the identity map of $B$. Composition is given by composition of maps among tensor powers of $B$.
\item 
The category of modules is a full subcategory of the category of cooperads (or operads) where for a module $M$, the cooperad $\MM$ has $M_0=M$, $M_1=I_1$, and only trivial compositions.
\item 
The coassociative cooperad has $M_n$ equal to the ground field for all $n$ with every decomposition map induced by the canonical isomorphism between the ground field and its tensor powers.
\end{itemize}
\end{example}
\begin{defi}
Operads have a forgetful functor to collections whose left adjoint is called the \emph{free} operad on a collection.

Conilpotent coaugmented cooperads have a forgetful functor to collections whose right adjoint is called the \emph{cofree} cooperad on a collection (suppressing conilpotence and coaugmentation).
\end{defi}
Both the cofree and free functor on $M$ can be realized at the collection level as the collection of rooted planar trees with vertices decorated by elements of $M$, denoted $\Tt(M)$. This implies the following.
\begin{enumerate}
\item
Fix an operad $\PP$ (with underlying collection $P$) and an element of $\Tt(M)$. That is, take a planar rooted tree $T$ and an element of $\PP(n)$ for every vertex of $T$ of valence $n+1$ (collectively called \emph{a decoration of $T$ by $\PP$}). Then there is a canonical element of $\PP$ called \emph{the composition of the decoration} induced by the counit of the forgetful free adjunction $\free(P)\xrightarrow{\boldsymbol{\epsilon}} \PP$. Since this is a monad, this operation is associative, in the sense that this composition can be done subtree by subtree and the output is insensitive to the choice of subtrees or order of composition.
\item Dually, given a cooperad $\CC$ with underlying collection $C$, the unit of the cofree forgetful adjunction $\CC\xrightarrow{\boldsymbol{\eta}} \cofree(C)$ yields the following. For every planar rooted tree $T$ with $n$ leaves and vertices $\{v_i\}$ where $v_i$ has valence $n_i+1$,  and every element $c\in \CC(n)$, there is a canonical set of elements $c_i\in \CC(n_i)$. This procedure is called the \emph{decomposition} of $c$ into a decoration of $T$ by $\CC$. It can be realized as follows. Let $\widetilde{\Delta}^N c$ stabilize as in the definition of conilpotence. Each summand corresponds to a tree with levels and decorations. Forget the levels and any decorations by $\II$ and project onto the summand corresponding to the tree $T$. See~\cite[5.8.7]{LodayVallette:AO} for more details.
\item
There is a canonical linear isomorphism $\psi$ between the free operad on a collection and the conilpotent cofree cooperad on the same collection.
\end{enumerate}
\begin{lemma}\label{lemma:isomorphismsmatch}
Let $M$ be a collection. An endomorphism $\cofree(M)\to\cofree(M)$ is an isomorphism if and only if its restriction $M\subset \cofree(M)\to M$ is an isomorphism of collections.
\end{lemma}
This is in precise parallel to the situation with power series, where a power series is invertible if and only if its constant term is invertible.
\begin{proof}
For $F$ an endomorphism, let $F_r$ denote its restriction. Note that $(F\circ G)_r = F_r\circ G_r$, which implies that if $F$ is an isomorphism, so is $F_r$. On the other hand, if $F_r$ is an isomorphism, then induction on the number of vertices in a tree in $\cofree(M)$ allows one to build an inverse $F^{-1}$.
\end{proof}
\begin{defi}\label{defi:canonical twist}
Let $\PP$ be an operad with underlying collection $P$ The \emph{canonical twist} $\Phi_\PP:\cofree(P)\to\cofree(P)$ is the cooperad map induced by the composition $\phi_\PP=\boldsymbol{\epsilon}\circ \psi$:
\[\cofree(P)\xrightarrow{\psi} \free(P)\xrightarrow{\boldsymbol{\epsilon}} \PP\to P.\]
\end{defi}
\begin{lemma}
Let $\PP$ be an operad. Then the canonical twist is an isomorphism.
\end{lemma}
\begin{proof}
Restricted to $P$, the canonical twist is the identity. Then Lemma~\ref{lemma:isomorphismsmatch} implies the result.
\end{proof}
\section{Partitions and trees}\label{sec: partitions and trees}
\begin{defi}
Let $[n]$ be an ordered set and let $\pi=(p_1,\ldots, p_k)$ be a partition of it, so that $[n]$ is the disjoint union of the blocks $p_i$. Blocks in our partitions are always ordered so that $\min p_i<\min p_j$ whenever $i<j$. A partition $\pi$ is \emph{crossing} if there exist $w$ and $y$ in $p_i$ and $x$ and $z$ in $p_j$ (with $i\ne j$) such that $w<x<y<z$. A partition $\pi$ is \emph{non-crossing} if it is not crossing. The notation $NC(n)$ (respectively $NC_k(n)$) refers to the set of noncrossing partitions of $[n]$ (with $k$ blocks). The unique partition with a single block is called the trivial partition.
\end{defi}
Noncrossing partitions are important in combinatorics and there are many bi-indexed sets of combinatorial objects in canonical bijection with them. 
For our purposes, the following such bijection will be useful.
\begin{lemma}
The set $NC_k(n)$ is in bijection with the set of planar rooted trees with $n$ leaves and $k+1$ vertices (including the root) satisfying the conditions that
\begin{enumerate}
\item Every non-root vertex has at least one leaf attached to it, and
\item the root has no leaves attached to it.
\end{enumerate}
The bijection from trees to partitions is given explicitly by numbering the leaves clockwise starting from the root and then letting two numbers share a block if the corresponding leaves are incident on the same vertex. Thus blocks are in bijection with non-root vertices.
\end{lemma}
See Figure~\ref{figure: partitions trees}. Henceforth partitions will be freely identified with the corresponding trees.
\begin{figure}
\begin{tikzpicture}
\draw (0,0) -- (0,-1);
\draw (0,0) -- (0,1);
\draw (2.5,2) -- (0,1);
\draw (1.5,2) -- (0,1);
\draw (.5,2) -- (0,1);
\draw (-.5,2) -- (0,1);
\draw (-1.5,2) -- (0,1);
\draw (-2.5,2) -- (0,1);
\draw (-1.5,2) -- (-1.5,3);
\draw (-.5,2) -- (-.5,3);
\draw (1.5,2) -- (1.5,3);
\draw [fill] (0,0) circle [radius=.05];
\draw [fill] (0,1) circle [radius=.05];
\draw [fill] (-1.5,2) circle [radius=.05];
\draw [fill] (-.5,2) circle [radius=.05];
\draw [fill] (1.5,2) circle [radius=.05];
\node [above] at (-2.5,2) {$1$};
\node [above] at (-1.5,3) {$2$};
\node [above] at (-.5,3){$3$};
\node [above] at (.5,2){$4$};
\node [above] at (1.5,3) {$5$};
\node [above] at (2.5,2) {$6$};
\node [below] at (0,-1) {$(146)(2)(3)(5)$};
\draw (6,0) -- (6,-1);
\draw (6,0) -- (7,1);
\draw (6,0) -- (5,1);
\draw (5,1) -- (6,2);
\draw (5,1) -- (5,2);
\draw (5,1) -- (4,2);
\draw (5,2) -- (4,3);
\draw (5,2) -- (5,3);
\draw (5,2) -- (6,3);
\draw (5,3) -- (5,4);
\draw (7,1) -- (7,2);
\draw [fill] (6,0) circle [radius=.05];
\draw [fill] (5,1) circle [radius=.05];
\draw [fill] (5,2) circle [radius=.05];
\draw [fill] (5,3) circle [radius=.05];
\draw [fill] (7,1) circle [radius=.05];
\node [above] at (4,2) {$1$};
\node [above] at (4,3) {$2$};
\node [above] at (5,4) {$3$};
\node [above] at (6,3) {$4$};
\node [above] at (6,2) {$5$};
\node [above] at (7,2) {$6$};
\node [below] at (6,-1) {$(15)(24)(3)(6)$};
\end{tikzpicture}
\caption{Two non-crossing partitions and the corresponding planar rooted trees}
\label{figure: partitions trees}
\end{figure}
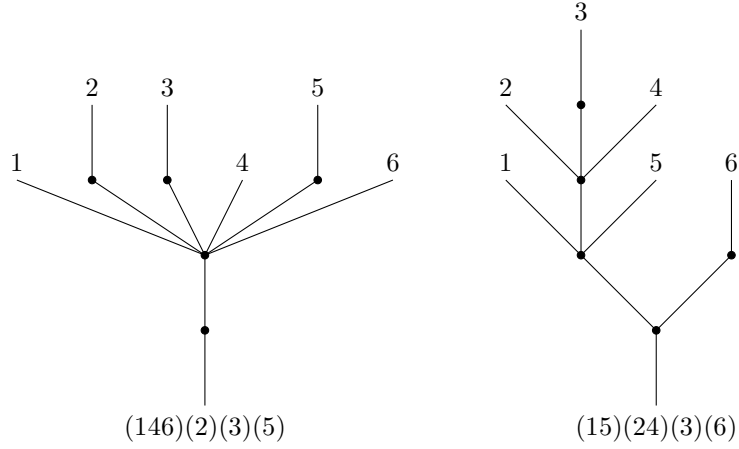
It will be useful later to modify this construction.
\begin{construction}\label{construction: messy tree}
Let $\pi$ be a noncrossing partition of $[n]$. There is a function $h:\{0,\ldots,n\}\to \{*,B_1,\ldots B_k\}$ defined by letting $h(i)$ be the maximal block (should one exist), which contains elements $x$ and $y$ of $[n]$ such that $x\le i$ and $y>i$. Should no such block exist, then $h(i)=*$.

Now let $i_0,\ldots, i_n$ be non-negative numbers. The tree $\pi_{i_0,\ldots,i_n}$ is obtained from $\pi$ by the following procedure:
\begin{enumerate}
\item For each $j$, attach $i_j$ new leaves at the vertex $h(j)$ of $\pi$ (if $h(j)=*$, attach to the root vertex) in the unique possible way so that the new leaves are after the $j$th original leaf and before the $j+1$st original leaf of $\pi$, using the clockwise order around leaves.
\item Consider a non-root vertex $v$. Its incoming half-edges are of the form $e_{0,1},\ldots,e_{0,k_0},\ell_1,e_{1,1},\ldots, e_{1,k_1},\ell_2,\ldots, \ell_r, e_{r_1,\ldots, e_{r,k_r}}$, where $\ell_i$ are original leaves and $e_{i,j}$ are new leaves or parts of edges.

Let $S_-$ be the set of indices $j$ in $\{0,\ldots, r\}$ such that $k_j\ge 1$ and let $S_+$ be the set of indices $j$ such that $k_j>1$. Now replace $v$ with a tree which has 
\begin{enumerate}
\item one ``bottom'' vertex with incoming half-edges in ordered bijection with $S_-$ and 
\item \label{item: bottom top} ``top'' vertices in ordered bijection with $S_+$ where the vertex $v_j$ has $k_j$ incoming half-edges.
\end{enumerate}
Join the root of a top vertex with the corresponding incoming half-edge of the bottom vertex; the other incoming half-edges of the bottom vertex and all incoming half-edges of the top vertex are identified with the incoming half-edges of $v$. 
\item Delete all of the original leaves of $\pi$; delete the root if it is bivalent at this point in the construction.
\end{enumerate}
The tree obtained after the intermediate step~\ref{item: bottom top} will also be important and will be called $\overline{\pi}_{i_0,\ldots, i_n}$.
\end{construction}
See Figure~\ref{fig: messy tree}.
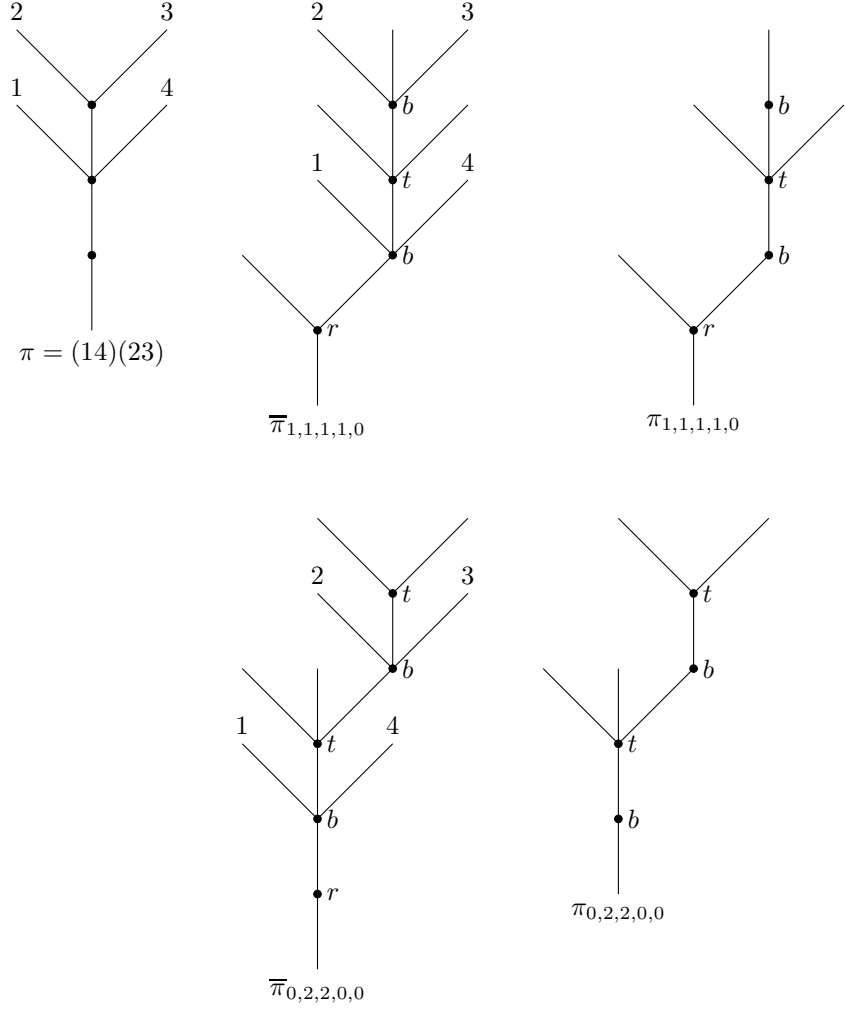
\begin{figure}
\begin{tikzpicture}
\draw(0,0)--(0,1);
\draw (0,1)--(0,2);
\draw(0,2)--(-1,3);
\draw(0,2)--(1,3);
\draw(0,2)--(0,3);
\draw(0,3)--(1,4);
\draw(0,3)--(-1,4);
\draw [fill] (0,1) circle [radius=.05];
\draw [fill] (0,2) circle [radius=.05];
\draw [fill] (0,3) circle [radius=.05];
\node[below] at (0,0){$\pi=(14)(23)$};
\node[above] at (-1,3){$1$};
\node[above] at (-1,4){$2$};
\node[above] at (1,4){$3$};
\node[above] at (1,3){$4$};
\begin{scope}[shift ={(3,4)}]
\draw(0,-5)--(0,-4);
\draw (0,-4)--(1,-3);
\draw  (0,-4)--(-1,-3);
\draw  (1,-3)--(0,-2);
\draw(1,-3)--(1,-2);
\draw  (1,-3)--(2,-2);
\draw(1,-2)--(2,-1);
\draw(1,-2)--(0,-1);
\draw(1,-2)--(1,-1);
\draw(0,0)--(1,-1);
\draw(1,0)--(1,-1);
\draw(2,0)--(1,-1);
\draw [fill] (0,-4) circle [radius=.05];
\draw [fill] (1,-3) circle [radius=.05];
\draw [fill] (1,-2) circle [radius=.05];
\draw [fill] (1,-1) circle [radius=.05];
\node[right] at (0,-4){$r$};
\node[right] at (1,-3){$b$};
\node[right] at (1,-2){$t$};
\node[right] at (1,-1){$b$};
\node[below] at (0,-5){$\overline{\pi}_{1,1,1,1,0}$};
\node[above] at (0,-2){$1$};
\node[above] at (0,0){$2$};
\node[above] at (2,0){$3$};
\node[above] at (2,-2){$4$};
\end{scope}
\begin{scope}[shift ={(8,4)}]
\draw(0,-5)--(0,-4);
\draw (0,-4)--(1,-3);
\draw  (0,-4)--(-1,-3);
\draw (1,-3)--(1,-2);
\draw  (1,-2)--(0,-1);
\draw(1,-2)--(1,-1);
\draw  (1,-2)--(2,-1);
\draw(1,-1)--(1,0);
\draw [fill] (0,-4) circle [radius=.05];
\draw [fill] (1,-3) circle [radius=.05];
\draw [fill] (1,-2) circle [radius=.05];
\draw [fill] (1,-1) circle [radius=.05];
\node[right] at (0,-4){$r$};
\node[right] at (1,-3){$b$};
\node[right] at (1,-2){$t$};
\node[right] at (1,-1){$b$};
\node[below] at (0,-5){${\pi}_{1,1,1,1,0}$};
\end{scope}
\begin{scope}[shift ={(3,-8.5)}]
\draw(0,0)--(0,1);
\draw (0,1)--(0,2);
\draw(0,2)--(-1,3);
\draw(0,2)--(1,3);
\draw(0,2)--(0,3);
\draw(0,3)--(-1,4);
\draw(0,3)--(0,4);
\draw(0,3)--(1,4);
\draw (1,4)--(1,5);
\draw (1,4)--(0,5);
\draw (1,4)--(2,5);
\draw (1,5)--(2,6);
\draw(1,5)--(0,6);
\draw [fill] (0,1) circle [radius=.05];
\draw [fill] (0,2) circle [radius=.05];
\draw [fill] (0,3) circle [radius=.05];
\draw [fill] (1,4) circle [radius=.05];
\draw [fill] (1,5) circle [radius=.05];
\node[right] at (0,1){$r$};
\node[right] at (0,2){$b$};
\node[right] at (0,3){$t$};
\node[right] at (1,4){$b$};
\node[right] at (1,5){$t$};
\node[below] at (0,0){$\overline{\pi}_{0,2,2,0,0}$};
\node[above] at (-1,3){$1$};
\node[above] at (0,5){$2$};
\node[above] at (2,5){$3$};
\node[above] at (1,3){$4$};
\end{scope}
\begin{scope}[shift ={(7,-8.5)}]
\draw (0,1)--(0,2);
\draw(0,2)--(0,3);
\draw(0,3)--(-1,4);
\draw(0,3)--(0,4);
\draw(0,3)--(1,4);
\draw (1,4)--(1,5);
\draw (1,5)--(2,6);
\draw(1,5)--(0,6);
\draw [fill] (0,2) circle [radius=.05];
\draw [fill] (0,3) circle [radius=.05];
\draw [fill] (1,4) circle [radius=.05];
\draw [fill] (1,5) circle [radius=.05];
\node[right] at (0,2){$b$};
\node[right] at (0,3){$t$};
\node[right] at (1,4){$b$};
\node[right] at (1,5){$t$};
\node[below] at (0,1){${\pi}_{0,2,2,0,0}$};
\end{scope}
\end{tikzpicture}
\caption{Examples of Construction~\ref{construction: messy tree}. Top, bottom, and root vertices are labelled $t$, $b$, and $r$ respectively.}\label{fig: messy tree}
\end{figure}
\begin{remark*}
Note that the first and last incoming half-edge at each ``bottom'' vertex of $\overline{\pi}_{i_0,\ldots, i_n}$ are always original leaves of $\pi$.
\end{remark*}
\section{Free probability and operator-valued free cumulants}\label{sec:freeprob}
\begin{defi}
Let $B$ be an algebra. A \emph{$B$-valued probability space} consists of a pair $(A,E)$ where $A$ is a $B$-algebra and $E$ is a $B$-linear map, called \emph{expectation} $A\to B$ such that the composition $B\to A\to B$ is the identity. By abuse of notation, $E$ will usually be omitted.
\end{defi}
Classically $B$ is the ground field but for a general theory it is necessary to allow more general algebras, in particular, non-commutative algebras. To be precise, a $B$-algebra is a $B$-bimodule $A$ equipped with a product $A\otimes_B A\to A$ and a $B$-linear map $\eta:B\to A$ which respects the product structure. 

Let $A$ be a $B$-valued probability space and let $f:A^{\otimes_B n}\to B$ be a $B$-multilinear map. For an $(n+1)$-tuple $(i_0,\ldots, i_n)$ of non-negative integers with sum $N$, define a map
\[
f_{i_0,\ldots, i_n}: Hom(A^{\otimes n}, Hom(B^{\otimes N}, B))
\]
whose evaluation on $a_1\otimes a_n$ is given by the composition
\[
B^{\otimes N}\to 
B^{\otimes i_0}\otimes A\otimes B^{\otimes i_1}\otimes A\otimes\cdots\otimes A\otimes B^{\otimes i_n}
\to A^{\otimes_B n}\xrightarrow{f} B.
\]
Where the first map inserts $a_j$ in the $j$th $A$ place and the second map is given by repeated use of the $B$-bimodule structure on $A$.

The map $f_{i_0,\ldots, i_n}(a_1,\ldots, a_n)$ can be realized as the composition in $\End B$ along a decoration of the tree $\pi_{\min\{i_0,1\},\ldots, \min\{i_n,1\}}$, where $\pi$ is the trivial partition of $[n]$. Decorate ``top'' vertices and the root, should it exist, with the product in $B$ and decorate the single vertex corresponding to the single block of $\pi$ with $f_{\min\{i_0,1\},\ldots, \min\{i_n,1\}}(a_1,\ldots, a_n)$.

The following definition is Definition 2.1.1 of~\cite{Speicher:CTFPAOVFPT}, restricted to $B$-algebras. It has been reworded to use operadic language.
\begin{defi}\label{defi: multiplicative function}
Let $A$ be a $B$-algebra. For $n\ge 1$, let $f^{(n)}:A^{\otimes_B n}\to B$ be a $B$-linear map. Then the \emph{multiplicative function} 
\[\hat{f}:\bigcup_n NC(n)\times A^{\otimes_B n}\to B\]
is defined on $(\pi, a_1\otimes\cdots\otimes a_n)$ as the composition in $\End B$ along a decoration of the tree $\pi_{0,\ldots, 0}$. ``Top'' vertices and the root vertex of $\pi$, if it survives in $\pi_{0,\ldots,0}$, are decorated with the product in $B$. Let $v$ be a ``bottom'' vertex of $\overline{\pi}_{0,\ldots, 0}$ with ordered incoming half-edge set
\[
\ell_0,e_{1,1},\ldots,e_{1,k_1},\ell_1,\ldots,\ell_{k-1},e_{k,i_k},\ell_k
\]
where $\ell_i$ is an original leaf of $\pi$ which is numbered $n(\ell_i)$ in $\pi$ and $i_j$ are non-negative numbers. Then the corresponding ``bottom'' vertex of $\pi_{0,\ldots 0}$ is decorated with
\[f^{(k)}_{0,i_1,\ldots, i_k,0}(a_{n(\ell_0)}\otimes\cdots\otimes a_{n(\ell_k)}).\]
Then $\hat{f}(\pi,a_1\otimes\cdots\otimes a_n)$ is the composition of this decoration of $\pi_{0,\ldots, 0}$ in the operad $\End B$, viewed as an element of $(\End B)(0)\cong B$.
 \end{defi} 
See Figure~\ref{figure: multiplicative function}. 
 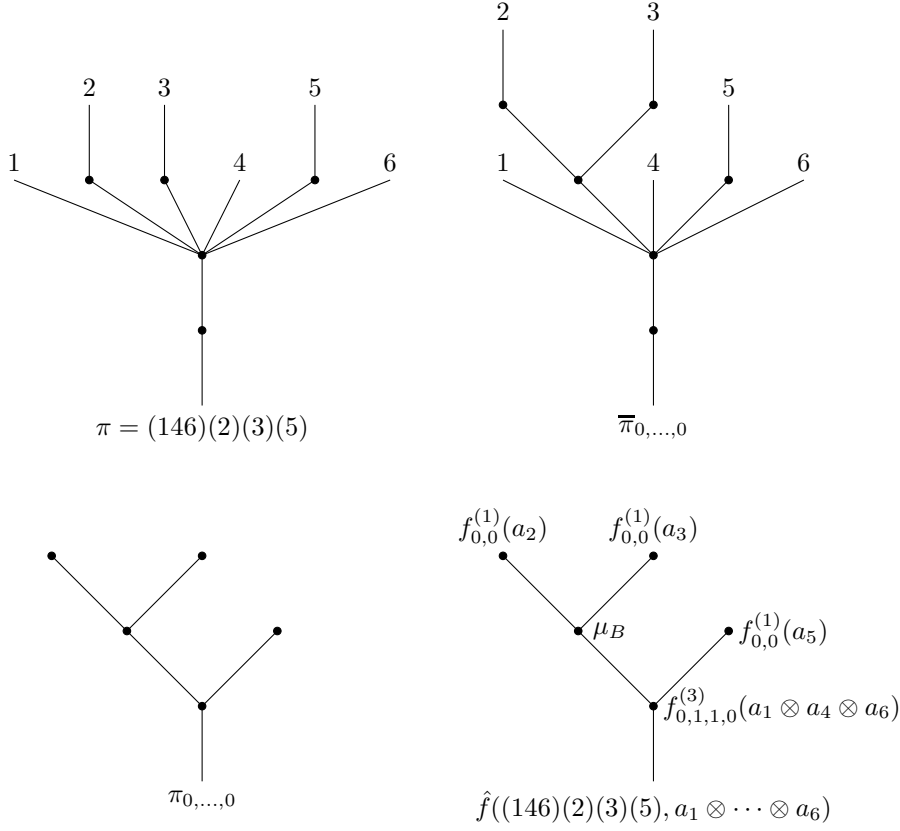
\begin{figure}
\begin{tikzpicture}
\draw (0,0) -- (0,-1);
\draw (0,0) -- (0,1);
\draw (2.5,2) -- (0,1);
\draw (1.5,2) -- (0,1);
\draw (.5,2) -- (0,1);
\draw (-.5,2) -- (0,1);
\draw (-1.5,2) -- (0,1);
\draw (-2.5,2) -- (0,1);
\draw (-1.5,2) -- (-1.5,3);
\draw (-.5,2) -- (-.5,3);
\draw (1.5,2) -- (1.5,3);
\draw [fill] (0,0) circle [radius=.05];
\draw [fill] (0,1) circle [radius=.05];
\draw [fill] (-1.5,2) circle [radius=.05];
\draw [fill] (-.5,2) circle [radius=.05];
\draw [fill] (1.5,2) circle [radius=.05];
\node [above] at (-2.5,2) {$1$};
\node [above] at (-1.5,3) {$2$};
\node [above] at (-.5,3){$3$};
\node [above] at (.5,2){$4$};
\node [above] at (1.5,3) {$5$};
\node [above] at (2.5,2) {$6$};
\node [below] at (0,-1) {$\pi=(146)(2)(3)(5)$};
\begin{scope}[shift = {(6,0)}]
\draw (0,0) -- (0,-1);
\draw (0,0) -- (0,1);
\draw (2,2) -- (0,1);
\draw (1,2) -- (0,1);
\draw (0,2) -- (0,1);
\draw (-1,2) -- (0,1);
\draw (-2,2) -- (0,1);
\draw (-1,2) -- (-2,3);
\draw (-1,2) -- (0,3);
\draw (0,3) -- (0,4);
\draw (-2,3) -- (-2,4);
\draw (1,2) -- (1,3);
\draw [fill] (0,0) circle [radius=.05];
\draw [fill] (0,1) circle [radius=.05];
\draw [fill] (-1,2) circle [radius=.05];
\draw [fill] (1,2) circle [radius=.05];
\draw [fill] (-2,3) circle [radius=.05];
\draw [fill] (0,3) circle [radius=.05];
\node [above] at (-2,2) {$1$};
\node [above] at (-2,4) {$2$};
\node [above] at (0,4){$3$};
\node [above] at (0,2){$4$};
\node [above] at (1,3) {$5$};
\node [above] at (2,2) {$6$};
\node [below] at (0,-1) {$\overline{\pi}_{0,\ldots, 0}$};
\end{scope}
\begin{scope}[shift = {(0,-6)}]
\draw (0,0) -- (0,1);
\draw (1,2) -- (0,1);
\draw (-1,2) -- (0,1);
\draw (-1,2) -- (-2,3);
\draw (-1,2) -- (0,3);
\draw [fill] (0,1) circle [radius=.05];
\draw [fill] (-1,2) circle [radius=.05];
\draw [fill] (1,2) circle [radius=.05];
\draw [fill] (-2,3) circle [radius=.05];
\draw [fill] (0,3) circle [radius=.05];
\node [below] at (0,0) {${\pi}_{0,\ldots, 0}$};
\end{scope}
\begin{scope}[shift = {(6,-6)}]
\draw (0,0) -- (0,1);
\draw (1,2) -- (0,1);
\draw (-1,2) -- (0,1);
\draw (-1,2) -- (-2,3);
\draw (-1,2) -- (0,3);
\draw [fill] (0,1) circle [radius=.05];
\draw [fill] (-1,2) circle [radius=.05];
\draw [fill] (1,2) circle [radius=.05];
\draw [fill] (-2,3) circle [radius=.05];
\draw [fill] (0,3) circle [radius=.05];
\node [right] at (0,1) {$f^{(3)}_{0,1,1,0}(a_1\otimes a_4\otimes a_6)$};
\node [right] at (1,2) {$f^{(1)}_{0,0}(a_5)$};
\node [right=2] at (-1,2) {$\mu_B$};
\node [above] at (-2,3) {$f^{(1)}_{0,0}(a_2)$};
\node [above] at (0,3) {$f^{(1)}_{0,0}(a_3)$};
\node [below] at (0,0) {$\hat{f}((146)(2)(3)(5),a_1\otimes\cdots\otimes a_6)$};
\end{scope}
\end{tikzpicture}
\caption{This figure demonstrates the evaluation of the multiplicative function $\hat{f}$ on the element 
\[(146)(2)(3)(5),a_1\otimes\cdots\otimes a_6.\] 
The eventual output is
\[f^{(3)}(a_1f^{(1)}(a_2)
f^{(1)}(a_3)\otimes a_4f^{(1)}(a_5)\otimes a_6).
\]
}
 \label{figure: multiplicative function}
 \end{figure}
The following definition combines Example 1.2.2, Definition 2.1.6, and Proposition 3.2.3 of~\cite{Speicher:CTFPAOVFPT}.
\begin{defi}
Let $A$ be a $B$-valued probability space. The free cumulant $\kappa_n:A^{\otimes_B n}\to B$ is defined recursively in terms of the expectation as follows:
\[
E(a_1\cdots a_n)=\sum_{\pi\in NC(n)} \hat{\kappa}(\pi,a_1\otimes\cdots\otimes a_n)
\]
\end{defi}
It will be useful in the next section to record a version of this defining relationship viewed in $\End B$. The following is a direct application of the definitions.
\begin{lemma}\label{lemma: moment-cumulant formula with insertions}
The moment and cumulant satisfy the following relations for nonnegative $i_0,\ldots, i_n$ and $a_1,\ldots, a_n$ in $A$:
\[
E_{i_0,\ldots, i_n}(a_1,\ldots, a_n) = \left(\sum_{\pi\in NC(n)}\hat{\kappa}(\pi,\bullet)\right)_{i_0,\ldots, i_n}(a_1,\ldots, a_n).
\]
which in turn is the composition in $\End B$ along the decorated tree $\pi_{i_0,\ldots,i_n}$ with decoration as in Definition~\ref{defi: multiplicative function}.
\end{lemma}
\section{Main result}\label{sec: main  result}
\begin{defi}
The cooperad $\coasA$ is the categorical product of the coassociative cooperad and the cooperad which is the algebra $A$ concentrated in arity $0$. 
\end{defi}
Let $V=\langle *\rangle$ be a one-dimensional free module. There is an explicit presentation of the sum of the modules of the underlying collection of $\coasA$ as $\bigoplus_{n=1}^\infty (A\oplus V)^{\otimes n}$. Here the arity $n$ module consists of those elements that are degree $n$ in the generator $*$ of $V$.

In this presentation, the cocomposition map is given as follows. Let $w$ be a word in $*$ and $A$; let $F(W)$ be the the set of all ways of writing $w$ as the concatenation $b_0a_1b_1\cdots a_nb_n$ where 
\begin{itemize}
\item The words $b_i$ are (possibly empty) words in $A$
\item The words $a_i$ are nonempty words in $*$ and $A$.
\end{itemize}
Then 
\[
\Delta w =\sum_{F(W)}(b_0*b_1*\cdots *b_n)\circ (a_1\otimes\cdots \otimes a_n).
\]
The projection map to $\coas$ is given by projecting to $\bigoplus V^{\otimes n}$ and identifying $*^{\otimes n}$ with $1$ in the ground ring. The projection map to $A$ is given by projecting to ${A\oplus V}^{\otimes 1}$, identifying $A$ with itself and $*$ with the image of $I$.
\begin{defi}
Let $A$ be a $B$-valued probability space. The \emph{moment morphism} $M:\coasA\to \cofree(\End B)$ is the map of cooperads determined by its linear restriction $m:\coasA\to \End B$ which is defined on the word 
\[
w=\underbrace{*\ldots *}_{i_0}a_1\underbrace{*\ldots *}_{i_1}a_2\ldots a_m\underbrace{*\ldots *}_{i_m}
\]
with $\sum i_j=n$ as 
\[
m(w)(b_1\otimes\cdots\otimes b_n)=(E\circ\mu_A)_{i_0,\ldots,i_m}(a_1\otimes\ldots\otimes a_m)
\]
or more explicitly
\[
m(w)(b_{0,1}\otimes\cdots\otimes b_{m,i_m})=
E(b_{0,1}\cdots b_{0,i_0}a_1b_{1,1}\cdots b_{1,i_1}a_2\cdots a_{m}b_{m,1}\cdots b_{m,i_m}).
\]
The \emph{free cumulant morphism} $K:\coasA\to \cofree(\End B)$ is the map of cooperads $\Phi^{-1}\circ M$.
\end{defi}

See Figure~\ref{fig: moment cumulant diagram}.
\begin{figure}\label{fig: moment cumulant diagram}
\[
\xymatrix{
	&\End B
	\\\\	
	&\cofree(\End B)\ar[uu]&\free(\End B)\ar@{=>}[uul]_{\boldsymbol{\epsilon}}
	\\
	\coasA\ar[uuur]^{m}\ar@{:>}[ur]_M\ar@{:>}[dr]^K\ar@{.>}[dddr]_k
	\\
	&\cofree(\End B)\ar@{:>}[uu]^\Phi\ar[uur]_{\psi}\ar[dd]
	\\\\
	&\End B
}
\]
\caption{This figure shows the relationship between moments and cumulants. In the diagram, single arrows are linear maps and double arrows are cooperad and/or operad maps.}
\end{figure}
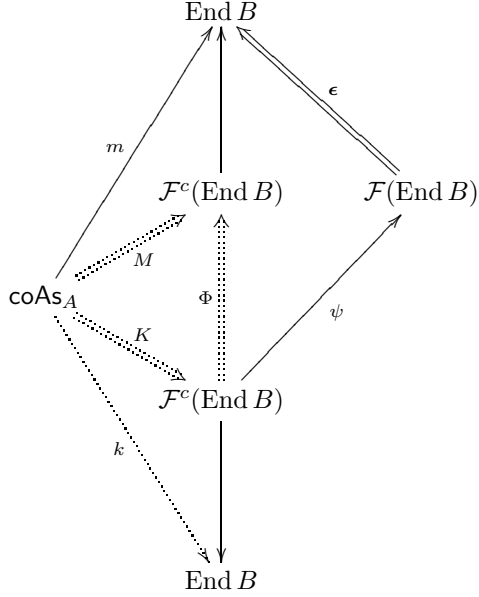

\begin{thm}[Main Result]
Let $A$ be a $B$-valued probability space. 

The restriction $k$ of the free cumulant morphism $K$ consists of the following:
\begin{enumerate}
\item The map $k(\underbrace{*\ldots*}_n)$ is $(-1)^n$ times multiplication $B^{\otimes n}\to B$ for $n>1$.
\item Let $\alpha_1,\ldots, \alpha_{m-1}$ be in $\{0,1\}$. Then 
\[
k(a_1\underbrace{*\ldots*}_{\alpha_1}a_2\ldots a_{m-2}\underbrace{*\ldots *}_{\alpha_{m-1}}a_m)
\]
(this description is slightly misleading because $\alpha_j\in \{0,1\}$) is
\[(\kappa_m)_{0,\alpha_1,\ldots,\alpha_{m-1},0}(a_1\otimes\cdots\otimes a_m)
\]
(in particular $k(a_1\ldots a_m)=\kappa(a_1\otimes\cdots \otimes a_m)$).
\item Applied to any word which contains an element of $A$ and also two consecutive $*$ or a word beginning or ending with $*$ which contains an element of $A$, the map $k$ vanishes.
\end{enumerate}
\end{thm}

\begin{proof}
For the duration of the proof, let $w$ be a word in $\overline{\coasA}(n)$ which contains precisely the $m$ letters (in order) $a_1,\ldots, a_m$ from $A$ and $n$ $*$ symbols, say
\[
w= \underbrace{*\ldots *}_{i_0}a_1\underbrace{*\ldots *}_{i_1}a_2\ldots a_m\underbrace{*\ldots *}_{i_m}
 \] 
with $i_j\ge 0$ and $\sum i_j=n$.

We give this word weight $2m-1+ \sum i_j$. Weight is nonnegative, positive on the cokernel of the coaugmentation, and preserved by the decomposition map of $\coasA$. The proof will proceed by induction on weight in each case.

Using the equivalent characterization $M=\Phi\circ K$ gives a recursive definition of $k$ on the weight $L$ component in terms of the value of $m$ on the weight $L$ component and the value of $k$ on components of strictly smaller weight. That is, since the codomain $\cofree(\End B)$ is cofree, it suffices to project to $\End B$ for the definition, which yields $m=\phi\circ K$. Then via~\cite[Prop. 5.8.6]{LodayVallette:AO} (or more properly speaking, its nonsymmetric version), the map $K$ can be written at the level of collections as the composition 
\[
\coasA\xrightarrow{\boldsymbol{\eta}}
\Tt(\coasA)
\xrightarrow{\Tt(k)}
\Tt(\End B)
\]
Since $\phi$ is just $\boldsymbol{\epsilon}\circ \psi$, and $\psi$ is the identity at the level of collections, we can the relationship between $m$ and $k$ as the commutativity of the following diagram:
\begin{equation}\label{equation: square}
\begin{gathered}
\xymatrix{
	\coasA\ar[r]^m\ar[d]_{\boldsymbol{\eta}}&\End B\\
	\Tt(\coasA)\ar[r]_{\Tt(k)}&\Tt(\End B)\ar[u]_{\boldsymbol{\epsilon}}	
}
\end{gathered}
\end{equation}
where $\boldsymbol{\eta}$ and $\boldsymbol{\epsilon}$ are the canonical decomposition and composition maps.

Now $\boldsymbol{\eta}(w)$ consists of a sum of trees decorated with elements of $\overline{\coasA}(j)$ for $j\le n$. There is one summand corresponding to a tree with a single vertex decorated by $w$ itself. This summand will contribute $k(w)$ to the eventual equation, and it is the purpose of the inductive step to determine its value. We will call this summand the \emph{trivial} summand and call the tree $T_{\triv}$.

We call decorated trees with a vertex decoration which contains an element of $A$ and either the string $**$ or the symbol $*$ at the beginning or end of the decoration \emph{degenerate}. In the first and second case of the statement of the theorem, by the inductive premise, only summands corresponding to nondegenerate decorated trees can contribute te $\boldsymbol{\eta}(w)$. In the third case, $T_{\triv}$ is the only degenerate decorated tree that may contribute.

Any nondegenerate decorated tree $T$ induces a partition $\pi_T$ of $[m]$ by saying $p$ and $q$ are in the same block if $a_p$ and $a_q$ are in the same vertex decoration. This partition is necessarily non-crossing because the original tree was planar. 

For $\pi$ a partition, let $\Tt_\pi$ be the set of nondegenerate decorated trees $T$ such that $\pi_T=\pi.$ Then the sum calculating $\boldsymbol{\eta}(w)$ polarizes into subsums:
\[
\sum_T \boldsymbol{\eta}(w)_T = \sum_{\pi\in NC(m)}  \sum_{T\in \Tt_{\pi}}\boldsymbol{\eta}(w)_T 
 \] 
in the first two cases in the statement of the theorem (in fact, in the first case $m$ is always $0$), and 
\[
\sum_T \boldsymbol{\eta}(w)_T = T_{\triv} + \sum_{\pi\in NC(m)}  \sum_{T\in \Tt_{\pi}}\boldsymbol{\eta}(w)_T 
 \] 
in the third case of the statement of the theorem.

Then given a partition $\pi$ arising from a nondegenerate decorated tree in the sum, the set of vertex decorations of $T\in \Tt_\pi$ which contain a letter of $A$ is independent of $T$. This set of vertex decorations can be recovered from $\pi$ as follows. Let $a_{i_0},\ldots, a_{i_j}$ be a block of $\pi$. Then necessarily the vertex decoration is of the form $a_{i_0}*^{\alpha_1}\cdots *^{\alpha_j} a_{i_j}$ where each $\alpha_p$ is either $0$ or $1$. If $a_{i_\ell}$ and $a_{i_{\ell+1}}$ are adjacent in $w$, then necessarily $\alpha_\ell=0$. On the other hand, if $a_{i_\ell}$ and $a_{i_{\ell+1}}$ are separated in $w$, then necessarily $\alpha_\ell=1$.

Fix a non-crossing partition $\pi$. Then there is a unique decorated tree in $T_\pi$ with a minimal number of vertices, obtained as a decoration of $\pi_{i_0,\ldots, i_n}$. The ``bottom'' vertices of this tree are decorated by the unique decoration described in the previous paragraph and all other vertices are decorated by $*\ldots *$.

Then it is straightforward to verify that the set of nongenerate decorated trees $T_\pi$ consists of all trees obtained from this decoration of $\pi_{i_0,\ldots, i_n}$ by replacing a vertex decorated by $*\ldots*$ by a tree of the same overall arity, all of whose vertices are at least trivalent, and all of whose vertices are decorated by $*\ldots*$.

At this point, it may be better to split into cases corresponding to the cases in the statement of the theorem.
\begin{enumerate}
\item On the word $w_n=(\underbrace{*\ldots*}_n)$ with $n>1$, the canonical decomposition $\boldsymbol{\eta}(w_n)$ is then the sum over all planar rooted trees with $n$ leaves; for each such tree the labels are all $*\ldots *$. On all nontrivial trees, $\boldsymbol{\epsilon}\circ \Tt(k)$ is, up to sign, just multiplication $\mu_n:B^{\otimes n}\to B$. Then by induction we have the formula
\[
\mu_n = m(w_n)=k(w_n)+ \sum \varepsilon_T \mu_n
\]
In fact the indexing set of trees $T$ for the sum is in canonical bijection with the non-top cells of the $n$-dimensional associahedron, and the sign $\varepsilon_T$ is just the dimension of the corresponding cell, so this is essentially a sum which calculates the Euler characteristic of the associahedron. The associahedron is contractible so we get
\[
\mu_n = k(w_n) + \mu_n(1- (-1)^n).
\]
\item Each individual set of trees $\Tt_\pi$ has its summands in bijection, as in the previous case, with faces of associahedra. To be precise, in this case there is a product of associahedra, one for each vertex of $\pi_{i_0,\ldots, i_n}$ decorated by $*\ldots *$. As in the previous case, the signs of $k$ applied to these decorations are such that after applying $\boldsymbol{\epsilon}$ to the subtrees where each vertex is decorated by $*\ldots *$, what is obtained is a redecoration of the tree $\pi_{i_0,\ldots, i_n}$, now by $\End B$, as follows.
\begin{enumerate}
\item Vertices that were previously decorated by $*\ldots *$ are now decorated by the product in $\End B$, with no sign.
\item Vertices that were previously decorated by 
\[a_{i_0}\underbrace{*\ldots *}_{\alpha_1}a_{i_1}\ldots a_{i_{j-1}}\underbrace{*\ldots*}_{\alpha_{j}}a_{i_j}\]
(this description is slightly misleading because $\alpha_j\in \{0,1\}$) are now decorated by induction by \[(\kappa_{j})_{0,\alpha_0,\ldots, \alpha_{j},0}(a_{i_0}\otimes\cdots \otimes a_{i_j})\] as in the statement of the theorem, except for the following special case.
\item the single tree $T_{\triv}$ is decorated by $k(w)$.
\end{enumerate}
Then by Lemma~\ref{lemma: moment-cumulant formula with insertions}, the equation $m(w)=\boldsymbol{\epsilon}\circ \Tt(k)\circ\boldsymbol{\eta}(w)$ is the same as the moment-cumulant formula for $E_{{i_0},{i_1},\ldots, {i_m},0}(a_1\otimes\cdots \otimes a_m)$, up to the difference
\[
k(w)-(\kappa_m)_{i_0,i_1,\ldots,i_m}(a_1\otimes\cdots\otimes a_m).
\]
So these two expressions are equal, as desired.
\item This is similar to the second case. Again, each set of trees $\Tt_\pi$ is in bijection with faces of products of associahedra and by the same trick one obtains a redecoration of $\pi_{i_0,\ldots, i_n}$. In this case the tree $T_{\triv}$ is not part of any $\Tt_\pi$ but instead is its own separate summand. Then in this case the equation $m(w)=\boldsymbol{\epsilon}\circ \Tt(k)\circ\boldsymbol{\eta}(w)$ is the same as the moment-cumulant formula for $E_{{i_0},{i_1},\ldots, {i_m},0}(a_1\otimes\cdots \otimes a_m)$, up to the difference
\[
k(w)
\]
so $k(w)$ is zero, as desired.
\end{enumerate}
\end{proof}
\subsection*{Concluding remarks}
\begin{remark*}
Equation~(\ref{equation: square}) in the preceding proof suggests a different interpretation of the main result. The maps $k$ and $m$ can be understood as maps of collections between the underlying collection of the cooperad $\coasA$ and the underlying collection of the operad $\End B$.  The space of maps of collections from a cooperad to an operad possesses a rich natural structure (see~\cite[6.4,10.2.3]{LodayVallette:AO} for details and notation). Apparently the relational equation can be expressed in terms of the convolution by the following expression:
\[m=\sum_{n=1}^\infty k^{\circledcirc n}\approx\frac{k}{1-k}.
\]
\end{remark*}
\begin{remark*}
The entire paper could be modified to work with symmetric operads; in this case it would be reasonable to replace $\coasA$ with a commutative version $\cocomA$. This would only make sense in the case that both $A$ and $B$ are commutative. As one might reasonably expect, this analagous procedure seems to describe the \emph{classical} cumulants as the canonical twist of the moments. As there are many direct combinatorial presentations~\cite{RotaShen:OCC} for classical cumulants and even a significantly more direct approach from the operadic point of view~\cite{DrummondColeParkTerilla:HPTII}, any details or verification have been omitted.
\end{remark*}

\bibliographystyle{amsalpha} 
\bibliography{references-2016}
\end{document}